\documentclass[a4paper]{amsart}
\usepackage{amsmath,mathrsfs}%
\usepackage{amsfonts}%
\usepackage{amssymb}%
\usepackage{graphicx}
\newtheorem{theorem}{Theorem}
\theoremstyle{plain}

\newtheorem{definition}{Definition}

\newtheorem{lemma}{Lemma}

\newtheorem{proposition}{Proposition}
\newtheorem{remark}{Remark}

\numberwithin{equation}{section}

\def\Xint#1{\mathchoice
   {\XXint\displaystyle\textstyle{#1}}%
   {\XXint\textstyle\scriptstyle{#1}}%
   {\XXint\scriptstyle\scriptscriptstyle{#1}}%
   {\XXint\scriptscriptstyle\scriptscriptstyle{#1}}%
   \!\int}
\def\XXint#1#2#3{{\setbox0=\hbox{$#1{#2#3}{\int}$}
     \vcenter{\hbox{$#2#3$}}\kern-.5\wd0}}

\def\dashint{\Xint-}

\newcommand{\smu}{{\mbox{\scriptsize$\mu$}}}

\newcommand{\dd}{\; \mathrm{d}}

\begin{document}
\title[Weak Gradients]{The $p$-Weak Gradient depends on $p$}
\author{Simone Di Marino}
\email[Simone Di Marino]{simone.dimarino@sns.it}
\author{Gareth Speight}
\email[Gareth Speight]{gareth.speight@sns.it}

\begin{abstract}
Given $\alpha>0$, we construct a weighted Lebesgue measure on $\mathbb{R}^{n}$ for which the family of non constant curves has $p$-modulus zero for $p\leq 1+\alpha$ but the weight is a Muckenhoupt $A_p$ weight for $p>1+\alpha$. In particular, the $p$-weak gradient is trivial for small $p$ but non trivial for large $p$. This answers an open question posed by several authors. We also give a full description of the $p$-weak gradient for any locally finite Borel measure on $\mathbb{R}$.
\end{abstract}
\maketitle

\section{Introduction}

Generalizations of Sobolev spaces to metric measure spaces is an important area of recent research \cite{HK00}, \cite{H07}. There are several different characterizations of classical Sobolev spaces that can be generalized to metric measure spaces; some, but not all, of these choices give rise to equivalent spaces. 

Typically, when defining a Sobolev space $W^{1,p}(X,d,m)$ on a metric measure space, a $p$-weak gradient $|\nabla f|_{m,p}$ of a Sobolev function $f\colon X \to \mathbb{R}$ is identified. In \cite{KM98} a $p$-weak gradient was defined as a minimal $p$-upper gradient (based on inequality in the Fundamental Theorem of Calculus along $p$-almost every curve, see Definition \ref{weakgradient}) while in \cite{AGS13} a $p$-weak gradient was defined by using relaxations of the slope in $L^{p}$. Remarkably, despite the fact that relaxation of slope does not explicitly involve curves, these definitions were shown to be equivalent in \cite{AGS13}.

The $p$-weak gradient agrees with the absolute value of the gradient for the classical case and, more generally, with the slope for Lipschitz functions defined on complete doubling metric measure spaces satisfying a weak $p$-Poincar\'{e} inequality. In the case of a general metric measure space the weak gradient is more subtle; an example suggested by P. Koskela shows that $|\nabla f|_{m,p} \in L^{q}(X,m)$ with $q>p$ does not imply that $|\nabla f|_{m,q}$ exists as a function in $L^{q}(X,m)$ \cite{AGS13}. Further, it was not clear whether, for a function in $W^{1,p}(X,d,m)\cap W^{1,q}(X,d,m)$ with $p\neq q$, the $p$-weak gradient and $q$-weak gradient agree (open problems 2.49, 2.53 \cite{BB11}, \cite{ACM13}, \cite{AGS13}).

We answer this question with the following theorem; we show that (even for a relatively nice measure on $\mathbb{R}^{n}$) the $p$-weak gradient may be trivial for small $p$ but non trivial for large $p$. We denote the $p$-modulus on absolutely continuous curves in a metric measure space $(X,d,m)$ by $\mathrm{Mod}_{p,m}$ (see Definition \ref{pmodulus}) and write $\mathbb{R}^{+}=[0,\infty)$.

\begin{theorem}\label{mainthm}
Let $n\in \mathbb{N}$ and $\alpha>0$. Then there exists a Borel function $w\colon \mathbb{R}^{n} \to \mathbb{R}^{+}$ such that the measure $\mu:=w\mathcal{L}^{n}$ is doubling and:
\begin{itemize}
	\item For $p\leq 1+\alpha$ we have $\mathrm{Mod}_{p,\mu}(\Gamma_{c})=0$ where $\Gamma_{c}$ is the family of non constant absolutely continuous curves in $\mathbb{R}^{n}$. This implies that the $p$-weak gradient on $(\mathbb{R}^{n},|\cdot|,\mu)$ is identically zero for every function.
	\item For $p>1+\alpha$ the function $w$ is a Muckenhoupt $A_{p}$-weight. This implies that a weak $p$-Poincar\'{e} inequality holds; it follows that the $p$-weak gradient on $(\mathbb{R}^{n},|\cdot|,\mu)$ agrees with the slope for Lipschitz functions.
\end{itemize}
\end{theorem}

The simple structure of curves in $\mathbb{R}$ gives rise to a simple description of the $p$-weak gradient with respect to each measure. In Theorem \ref{thm:char} we show that, for any locally finite Borel measure on $\mathbb{R}$ and $p>1$, the corresponding $p$-weak gradient of a Lipschitz function $f\colon \mathbb{R} \to \mathbb{R}$ is, at almost every point, equal to either zero or $|f'(x)|$. Roughly, the points where the $p$-weak gradient is non zero are those points which have a neighborhood that, when considered as a set containing a single curve, has positive $p$-modulus.

We now give some definitions and describe how the facts about weak gradients in Theorem \ref{mainthm} follow from the assertions about the measure. If $(X,d)$ is a metric space then a curve in $X$ is simply a continuous map $\gamma\colon I \to X$ where $I\subset \mathbb{R}$ is a closed interval; we denote the end points of a curve $\gamma$ by $a_{\gamma}$ and $b_{\gamma}$. In this paper we consider only absolutely continuous curves; a curve $\gamma\colon I \to X$ is absolutely continuous if there exists $g\colon I \to \mathbb{R}$ Lebesgue integrable such that
\[d(\gamma(s),\gamma(t))\leq \int_{s}^{t} g(r) \dd r\]
whenever $s, t \in I$ with $s<t$. If $\gamma$ is absolutely continuous then there is a minimal function $g$ with this property, called metric speed, which we denote by $|\dot{\gamma}|$. If $f\colon X \to \mathbb{R}$ is Borel then we may define
\[\int_{\gamma}f=\int_{0}^{1}(f(\gamma(s))) |\dot{\gamma}|(s) \dd s.\]

The $p$-modulus gives a way to measure the size of a family of curves in a metric measure space \cite{Hei01}. We consider only metric measure spaces $(X,d,m)$ for which $(X,d)$ is complete and separable with $m$ a $\sigma$-finite Borel measure on $X$.

\begin{definition}\label{pmodulus}
Let $(X,d,m)$ be a metric measure space and $p\geq 1$. The $p$-modulus $\mathrm{Mod}_{p,m}$ is an outer measure on the space of absolutely continuous curves in $X$ defined by
\[\mathrm{Mod}_{p,m}(\Gamma)=\inf \int_{X} g^{p} \dd m\]
where the infimum is taken over all Borel functions $g\colon X \to [0,\infty]$ satisfying
\[\int_{\gamma} g \dd s \geq 1\]
for all curves $\gamma \in \Gamma$. We say a property of curves holds for $p$-a.e. curve if the set of absolutely continuous curves for which it fails has $p$-modulus zero.
\end{definition}

The definition of $p$-weak gradient is based on inequality in the Fundamental theorem of calculus along $p$-a.e. curve \cite{KM98}, \cite{AGS13}.

\begin{definition}\label{weakgradient}
Let $(X,d,m)$ be a metric measure space and $p\geq 1$. A Borel function $g\colon X \to [0,\infty]$ is a $p$-upper gradient of $f\colon X \to \mathbb{R}$ if
\[|f(b_{\gamma})-f(a_{\gamma})| \leq \int_{\gamma} g \dd s \text{ for }p\text{-a.e. curve }\gamma.\]
If $p>1$ then the minimal $p$-upper gradient $|\nabla f|_{m, p}$ of $f\colon X \to \mathbb{R}$ is the $p$-upper gradient characterized, up to $m$-negligible sets, by the property
\[|\nabla f|_{m, p} \leq g \qquad m \text{-a.e. in }X \text{ for every $p$-upper gradient }g \text{ of }f.\]
\end{definition}

For the remainder of the paper we fix $\alpha>0$ and denote $\beta=1/\alpha$. Let $\mu$ be the measure from Theorem \ref{mainthm} and consider the metric measure space $(\mathbb{R}^{n},|\cdot|,\mu)$ so that $p\leq 1+\alpha$ implies $\mathrm{Mod}_{p,\mu}(\Gamma_{c})=0$. In this case the function identically equal to zero is a $p$-upper gradient for every function; hence $|\nabla f|_{m,p}=0$ for any function $f\colon \mathbb{R}^{n} \to \mathbb{R}$.

Now we recall the notion of a Muckenhoupt $A_p$-weight on $\mathbb{R}^{n}$; we only consider the case $p>1$ though a similar definition may be given for $p=1$ \cite{HKM93}. If $(X,d,m)$ is a metric measure space, with $f\colon X \to \mathbb{R}$ and $A\subset X$ Borel measurable such that $m(A)>0$, then we denote $f_{A}=\dashint_{A}f \dd m=(1/m(A))\int_{A}f \dd m$ whenever the quotient is well defined. If no measure is specified, integrals over subsets of $\mathbb{R}^{n}$ are with respect to Lebesgue measure $\mathcal{L}^{n}$; we also use the notation $\mathcal{L}^{n}(A)=|A|$.

\begin{definition}\label{Ap}
Let $p>1$. A function $w\colon \mathbb{R}^{n} \to \mathbb{R}^{+}$ is a Muckenhoupt $A_{p}$-weight if for some constant $C>0$ and all balls $B\subset \mathbb{R}^{n}$,
\begin{equation}\label{Apinequality}
\left(\dashint_{B}w\right) \left( \dashint_{B} w^{1/(1-p)} \right)^{p-1} \leq C.
\end{equation}
\end{definition}

Muckenhoupt $A_{p}$-weights were first introduced in \cite{Muc72} as precisely those weights for which the Hardy maximal function of the associated measure is bounded in $L^{p}$. The $A_{p}$ condition has numerous applications, for example to weighted Sobolev spaces \cite{CC94} and regularity of the solutions of degenerate elliptic equations \cite{FKS82}.

We recall that a Borel measure $m$ on a metric space $X$ is doubling if balls have finite positive measure and there is a constant $C\geq 1$ such that
\[m(B(x,2r))\leq Cm(B(x,r))\]
for all $x \in \mathbb{R}^{n}$ and $r>0$. The slope $|\nabla f|\colon X \to \mathbb{R}^{+}$ of a locally Lipschitz function $f\colon X \to \mathbb{R}$ is defined by
\[|\nabla f|(x)=\limsup_{y \to x}\frac{|f(y)-f(x)|}{d(x,y)}.\]
A Borel measure $m$ on $X$ admits a weak $p$-Poincar\'{e} inequality if there are constants $C>0$ and $\lambda \geq 1$ such that
\begin{equation}\label{poincareinequality}
\dashint_{B} |f-f_{B}|\dd m \leq Cr\left( \dashint_{\lambda B} |\nabla f|^{p}\dd m \right)^{1/p}
\end{equation}
whenever $B$ is a ball with radius $r$ and $f$ is a locally Lipschitz function on $\lambda B$. The notion of a Poincar\'{e} inequality on a metric measure space was originally introduced in \cite{HK98} to study quasiconformal mappings. Note that, by H\"{o}lder's inequality, the condition of a weak $p$-Poincar\'{e} inequality becomes weaker as $p$ increases. If a metric measure space equipped with a doubling measure admits a weak $p$-Poincar\'{e} inequality then it admits a differentiable structure \cite{Che99}; in fact, a Lip-lip inequality suffices in place of a Poincar\'{e} inequality \cite{Kei03}. Roughly, a Lip-lip inequality states that at almost every point the variation of a Lipschitz function on small scales is independent of the precise choice of scale.

We use the fact that if $w$ is a Muckenhoupt $A_{p}$-weight on $\mathbb{R}^{n}$ then the measure $\mu = w\mathcal{L}^{n}$ is $p$-admissible \cite{HKM93}; this means that $\mu$ is doubling and satisfies a weak $p$-Poincar\'{e} inequality. For $n=1$ the converse holds: if $\mu$ is $p$-admissible then $w$ must be an $A_{p}$-weight \cite{BBK05}. However, inequality \eqref{Apinequality} seemed easier to check than verifying inequality \eqref{poincareinequality} directly.

If a doubling metric measure space admits a weak $p$-Poincar\'{e} inequality then, for Lipschitz functions, the $p$-upper gradient $|\nabla f|_{m,p}$ agrees, up to negligible sets, with the slope \cite{Che99}, \cite{Kei03}. Hence, for $p>1+\alpha$, if $\mu$ is the measure in Theorem \ref{mainthm}, then the $p$-weak slope $|\nabla f|_{m,p}$ of Lipschitz functions $f\colon \mathbb{R}^{n} \to \mathbb{R}$ on $(\mathbb{R}^{n},|\cdot|,\mu)$ is non trivial.

We also note that, since $\mu$ is absolutely continuous with respect to Lebesgue measure, the metric measure space $(\mathbb{R}^{n},|\cdot|,\mu)$ satisfies a Lip-lip inequality. Further, in any metric measure space $(X,d,m)$, lower semicontinuity of the map, defined on Lipschitz functions,
\[f \mapsto \int_{X}|\nabla f|^{p} \dd m\]
in $L^{p}$ implies the $p$-weak gradient agrees with the slope for Lipschitz functions \cite{ACM13}. Hence we observe that a Lip-lip inequality is not sufficient for lower semicontinuity of the integral of the $p$-th power of the slope; this answers a question raised in \cite{ACM13}.

We now give an idea of the construction of the weight $w$ in Theorem \ref{mainthm}. Firstly we suppose $n=1$; one starts with the weight $w_{1}\equiv 1$, then repeatedly defines $w_{k}=\min \{w_{k-1},g_{k}\}$ where $g_{k}$ is a scaled and translated copy of $|x|^{\alpha}$ centred on some rational $q_{k}$. We do this for a dense, non repeating, sequence of rationals $(q_{k})_{k=1}^{\infty}$ and define $w=\inf_{k} w_{k}$. The function $1/w^{s}$ is locally integrable for $s<\beta$ but nowhere locally integrable for $s\geq \beta$; this discrepancy allows us to prove the first property in Theorem \ref{mainthm}. Further, provided the copies of $|x|^{\alpha}$ are scaled to be sufficiently thin, each stage in the construction increases the left hand side of inequality \eqref{Apinequality} only a small amount; this allows us to prove the second property in Theorem \ref{mainthm}. To prove Theorem \ref{mainthm} for general $n$ we define $\widehat{w}(x_{1},\dots,x_{n})=\min \{ w(x_{1}), \dots ,w(x_{n})\}$ on $\mathbb{R}^{n}$. Then $\widehat{w}$ has the same integrability properties as $w$ (but now with respect to $\mathcal{L}^{n}$), which gives the first property, and the lattice property of $A_{p}$-weights \cite{KKM} allows us to extend the second property from $w$ to $\widehat{w}$.

\textbf{Acknowledgement} The authors acknowledge the support of the grant ERC ADG GeMeThNES. The authors thank Luigi Ambrosio for highlighting the question, mentioned in \cite{AGS13}, which led to this paper and for helpful comments during the preparation of the paper. We also thank an anonymous referee for giving useful comments and for pointing out the extension of the example to $\mathbb{R}^{n}$ for $n>1$.

\section{Construction of the weight}

Fix a sequence $\varepsilon_{k}>0$ such that $\prod_{k=1}^{\infty}(1+\varepsilon_{k})<\infty$ and enumerate the rational numbers by a sequence $(q_{k})_{k=1}^{\infty}$ with $q_{k}\neq q_{l}$ for $k\neq l$. We inductively define a sequence of continuous weights $w_{k}\colon \mathbb{R} \to \mathbb{R}^{+}$; among other properties the weights satisfy $w_{k}\leq w_{k-1}$ and $w_{k}(x)>0$ if $x\notin \{q_{l}\colon l=1, \ldots, k\}$. Denoting by $w$ the limit of the weights $w_{k}$ we will verify Theorem \ref{mainthm} for the weight $\widehat{w}$ on $\mathbb{R}^{n}$ given by $\widehat{w}(x_{1},\dots,x_{n})=\min \{ w(x_{1}),\ldots,w(x_{n}) \}$.

Let $w_{1}\colon \mathbb{R} \to \mathbb{R}^{+}$ be the function which is constant and equal to $1$. Fix $k \in \mathbb{N}$ for which the weight $w_{k-1}$ has been defined; we show how to define $w_{k}$. Since $w_{k-1}$ is continuous and $w_{k-1}(q_{k})>0$ (using the properties described in the previous paragraph) we can choose $R_{k}>0$ so that 
\[w_{k-1}(q_{k})/2 \leq w_{k-1}(x) \leq 2w_{k-1}(q_{k})\]
for $|x-q_{k}|\leq 4R_{k}$. 

Fix $r_{k}>0$ such that:
\[r_{k}\leq w_{k-1}(q_{k})^{\beta}\varepsilon_{k},\]
\[8r_{k}\leq \varepsilon_{k}(R_{k}-r_{k})\]
and
\[2r_{k}(p-\alpha+1)/(p-1)\leq \varepsilon_{k}(R_{k}-r_{k}).\]
We let
\[g_{k}(x)=2w_{k-1}(q_{k})|(x-q_{k})/r_{k}|^{\alpha}\]
for $x \in \mathbb{R}$ and define $w_{k}\colon \mathbb{R} \to \mathbb{R}^{+}$ by
\[w_{k}(x)=\min \{w_{k-1}(x),g_{k}(x)\}.\]
The function $w_{k}$ is continuous, $w_{k}\leq w_{k-1}$ and $w_{k}>0$ if $x\notin \{q_{l}\colon l=1, \ldots, k\}$.

Denote $I_{k}=(q_{k}-r_{k},q_{k}+r_{k})$ and note that $w_{k}=w_{k-1}$ outside $I_{k}$. We also define $J_{k}=(q_{k}-R_{k},q_{k}+R_{k})$, $J_{k}^{+}=[q_{k}+r_k,q_{k}+R_{k})$ and $J_{k}^{-}=(q_{k}-R_{k},q_{k}-r_k)$. 

Let $w\colon \mathbb{R} \to \mathbb{R}^{+}$ be given by $w=\inf_{k}w_{k}$. We define a Borel weight $\widehat{w}\colon \mathbb{R}^{n} \to \mathbb{R}^{+}$ by:
\[\widehat{w}(x_{1},\dots,x_{n})=\min \{ w(x_{1}),\ldots,w(x_{n}) \} \]
and let $\mu=\widehat{w} \mathcal{L}^{n}$.

\section{The $p$-modulus on curves is trivial for small $p$}

In this section we show that $p<1+\alpha$ implies $\mathrm{Mod}_{p,\mu}(\Gamma_{c})=0$, where $\Gamma_{c}$ is the family of non constant absolutely continuous curves in $\mathbb{R}^{n}$. This fact arises from simple integrability properties of $1/w$ on $\mathbb{R}$ which follow from corresponding properties of $1/|x|^{\alpha}$. Recall that $\beta=1/\alpha$.

\begin{lemma}\label{integrability}
Let $r=e^{\alpha(\alpha+1)}$. The weight $w\colon \mathbb{R} \to \mathbb{R}^{+}$ has the following integrability properties:
\begin{itemize}
	\item[(1)] The function $1/w^{s}$ is locally Lebesgue integrable if $s<\beta$.
	\item[(2)] The function $1/(w^{\beta} |\log ( w/r ) |^{1+\alpha})$ is locally Lebesgue integrable.
	\item[(3)] The function $1/w^{s}$ is nowhere locally Lebesgue integrable if $s\geq \beta$.
	\item[(4)] The function $1/(w^{\beta} |\log( w/r ) |)$ is nowhere locally Lebesgue integrable.
\end{itemize}
\end{lemma}

\begin{proof}
Suppose first $s<\beta$ and $N \in \mathbb{N}$. Clearly, for each integer $k>1$, $w_{k}=w_{k-1}$ outside $I_{k}$ implies
\[\int_{-N}^{N} \frac{1}{w_{k}^{s}}\leq \int_{-N}^{N} \frac{1}{w_{k-1}^{s}}+\int_{q_{k}-r_{k}}^{q_{k}+r_{k}}\frac{1}{w_{k}^{s}}.\]
We show the second term is relatively small. Indeed, since
\[w_{k}(x)\geq \frac{1}{2}w_{k-1}(q_{k})|(x-q_{k})/r_{k}|^{\alpha}\]
for $x \in (q_{k}-r_{k},q_{k}+r_{k})$ and $\alpha s < 1$, we have,
\begin{align*}
\int_{q_{k}-r_{k}}^{q_{k}+r_{k}}\frac{1}{w_{k}^{s}}&\leq \frac{2^{s}r_{k}^{\alpha s}}{w_{k-1}(q_{k})^{s}}\int_{q_{k}-r_{k}}^{q_{k}+r_{k}}\frac{1}{|x-q_{k}|^{\alpha s}}\\
&\leq Cr_{k}/w_{k-1}(q_{k})^{s}\\
&\leq Cr_{k}/w_{k-1}(q_{k})^{\beta}\\
&\leq C\varepsilon_{k}.
\end{align*}
Since $w_{1}$ was constant (so trivially locally integrable) and $\varepsilon_{k}$ were chosen small we deduce that the sequence $\int_{-N}^{N} 1/w_{k}^{s}$ is bounded uniformly in $k$. By the Monotone Convergence Theorem we obtain that $1/w^{s}$ is integrable on the interval $[-N,N]$.

For the second assertion a similar estimate is required: first of all the function $\Phi : t \mapsto t (-\log(t^{\alpha}/r) )^{1+\alpha}$ is increasing in $(0,1)$, and thus we can make the estimate
$$\int_{q_k-r_k}^{q_k+r_k} \frac 1{\Phi (w_k^{\beta}) } \leq \int_{q_k-r_k}^{q_k+r_k} \frac 1{ \Phi ( C_k |\frac{x-q_k}{r_k}| )} = \frac {2r_k}{C_k} F(C_k)$$
where $C_k = ( w_{k-1}(q_k)/2 )^{\beta}$ and $F$ is the primitive of $1/\Phi$ such that $F(0)=0$. Substituting  $r_{k}\leq CC_{k}\varepsilon_{k}$ and using the definition of $\Phi$ we obtain
$$\int_{q_k-r_k}^{q_k+r_k} \frac 1{ w_k^{\beta} | \log(w_k /r)| ^{1+\alpha} }  \leq \frac {2r_k}{C_k} F(1) \leq C \varepsilon_k.$$
We now obtain the required integrability as before.

Now suppose $s\geq \beta$ and $I$ is a non empty interval. Then we can find $k \in \mathbb{N}$ for which $q_{k} \in I$. It follows,
\[\int_{I} 1/w^{s} \geq C(k)\int_{I} 1/|x-q_{k}|^{\alpha s}\]
and the right hand side is equal to $\infty$ since $\alpha s \geq 1$. In the same way we have that $w_k^\beta \log ( w_k/r) \sim C(k) |x-q_k| \log |x-q_k|$ in a neighborhood of $q_k$ and so the final statement follows.
\end{proof}

Notice the previous lemma implies that $w$ is nonzero outside a set of Lebesgue measure zero. We recall some elementary facts about the modulus which are valid on any metric measure space \cite{Hei01}.

\begin{lemma}\label{modulusfacts}
Let $(X,d,m)$ be a metric measure space. The modulus $\mathrm{Mod}_{p,m}$ satisfies
\[\mathrm{Mod}_{p,m}(\Gamma_{a})\leq \mathrm{Mod}_{p,m}(\Gamma_{b})\]
if $\Gamma_{a}$ and $\Gamma_{b}$ are two curve families such that each curve in $\Gamma_{a}$ has a subcurve in $\Gamma_{b}$. Further, $\mathrm{Mod}_{p,m}(\Gamma)=0$ if and only if there is a $p$-integrable Borel function $g\colon X\to [0,\infty]$ such that $\int_{\gamma} g \dd s=\infty$ for each $\gamma \in \Gamma$.
\end{lemma}

Now we can deduce the required properties of the $p$-modulus on $(\mathbb{R}^{n},|\cdot|,\mu)$.

\begin{proposition}\label{trivialmodulusprop}
Let $\Gamma_{c}$ be the family of non constant absolutely continuous curves on $\mathbb{R}^{n}$ and $p\leq 1+\alpha$. Then $\mathrm{Mod}_{p,\mu}(\Gamma_{c})=0$.
\end{proposition}

\begin{proof}
For each $k \in \mathbb{N}$ let $\Gamma_{k}$ be the family of non constant absolutely continuous curves with image contained in $[-k,k]^{n}$. Using Lemma \ref{modulusfacts} it suffices to show that $\mathrm{Mod}_{p,\mu}(\Gamma_{k})=0$ for each $k$.

First suppose $p<1+\alpha$; fix $k \in \mathbb{N}$ and recall $\beta=1/\alpha$. Let $g\colon \mathbb{R}^{n} \to \mathbb{R}^{+}$ be equal to $1/\widehat{w}^{\beta}$ inside $[-k,k]^{n}$ and identically $0$ outside $[-k,k]^{n}$. Suppose $\gamma = (\gamma_{1},\dots,\gamma_{n})  \in \Gamma_{k}$ and fix $i$ such that the image of $\gamma_{i}$ contains some non trivial interval $I\subset \mathbb{R}$. Then,
\begin{align*}
\int_{\gamma} g\dd s &\geq \int_{\gamma} 1/w(x_{i})^{\beta} \dd s\\
&\geq \int_{\gamma_{i}} 1/w(t)^{\beta} \dd s\\
&\geq \int_{I} 1/w(t)^{\beta} \dd t\\
&= \infty
\end{align*}
using Lemma \ref{integrability}. However,
\begin{align*}
\int_{\mathbb{R}^{n}} g^{p} \dd \mu &= \int_{[-k,k]^{n}} \widehat{w}^{1 - \beta p}\\
&\leq \int_{[-k,k]^{n}} \sum_{i=1}^{n} w(x_{i})^{1-\beta p}\\
&\leq n(2k)^{n-1}\int_{-k}^{k} w(t)^{1-\beta p} \dd t
\end{align*}
which, by Lemma \ref{integrability}, is finite if $\beta p -1 <\beta$ or, equivalently, $p<1+\alpha$. Hence, by Lemma \ref{modulusfacts}, $\mathrm{Mod}_{p,\mu}(\Gamma_{k})=0$ and the proposition follows.

In the case $p=1+\alpha$ we choose $g=1/(\widehat{w}^{\beta}|\log(\widehat{w}/r)|)$; the argument is then identical using the analogous statements about integrability from Lemma \ref{integrability}.
\end{proof}

\section{The Muckenhoupt $A_{p}$ condition for large $p$}

We suppose throughout this section that $p>1+\alpha$. We first show that $w$ is a Muckenhoupt $A_{p}$-weight on $\mathbb{R}$ and then deduce $\widehat{w}$ is an $A_{p}$-weight on $\mathbb{R}^{n}$ using the lattice property of $A_{p}$-weights \cite{KKM}. To verify $w$ is a Muckenhoupt $A_{p}$-weight the idea will be that constructing $w_{k}$ from $w_{k-1}$ can increase the left side of inequality \eqref{Apinequality} only very slightly. We use a different argument depending on whether the ball in \eqref{Apinequality} is relatively small or relatively large.

It will be important during the proof that $|x|^{\alpha}$ is a Muckenhoupt $A_p$-weight on $\mathbb{R}$; this fact is well known (for example see Remark 4 \cite{BBK05}; this is also valid in $\mathbb{R}^n$ provided $p>1+n\alpha$) but we prefer to provide here a self-contained proof.

\begin{lemma}\label{lem:xalpha} The function $g(x)=|x|^{\alpha}$ on $\mathbb{R}$ is an $A_p$-weight. 
\end{lemma}

\begin{proof} Let $I=[a,b]$ be an interval. Denote $I^+ = I \cap [0,+\infty)$ and $I^-=I \cap ( -\infty,0]$. Without loss of generality we can assume that $|I^+| \geq |I^-|$; in this case we have that $I^- \subseteq - I^{+}$ and so, using that $g$ is an even function, we have
\[ \left(  \dashint_I g \right) \left( \dashint_I g^{1/(1-p)} \right)^{p-1} \leq 2^p \left(  \dashint_{I^+} g \right) \left( \dashint_{I^+} g^{1/(1-p)} \right)^{p-1}. \]
Hence it is sufficient to prove \eqref{Apinequality} only for intervals $I=[a,b]$ such that $0\leq a < b$.\\
We distinguish two cases:
\begin{itemize}
 \item $2a \geq b$. In this case, given the monotonicity of $g$ we can estimate each of the factors in the left hand side of \eqref{Apinequality} with the values of the integrand at the endpoint: in particular we can estimate it from above by $g(b)/g(a) \leq 2^{\alpha}$.
 
 \item $2a < b$. In this case we have that $1/(b-a) \leq 2/b$ and so
\[ \dashint_a^b x^{\alpha} \, \dd x \leq \frac 1{b-a} \int_0^b x^{\alpha} \, \dd x \leq \frac {2b^{\alpha}}{\alpha+1}; \]
\[ \dashint_a^b x^{\alpha/(1-p)} \, \dd x \leq \frac 1{b-a} \int_0^b x^{\alpha/(1-p)} \, \dd x \leq \frac {2b^{\alpha/(1-p)}}{\alpha/(1-p)+1}. \]
 These two inequalities together give us precisely \eqref{Apinequality}, with $C$ depending only on $\alpha$ and $p$.
\end{itemize} 
\end{proof}

The following Lemma will be used to estimate \eqref{Apinequality} for relatively small intervals; the idea will be that early stages in the construction play no role on small scales.

\begin{lemma}\label{universal}
Suppose $q \in \mathbb{R}$, $R>0$ and $f\colon (q-R,q+R) \to \mathbb{R}^{+}$ is Borel with $L/2\leq f\leq 2L$ for some $L>0$. 

Let $0<r<R$ and $g(x)=2L|(x-q)/r|^{\alpha}$ for $x \in \mathbb{R}$. 

Define $h\colon (q-R,q+R)\to \mathbb{R}^{+}$ by
\[h(x)=\min\{f(x),g(x)\}.\]
Then for any interval $I\subset (q-R,q+R)$ we have,
\begin{equation}\label{localizedAp}
\left(\dashint_{I}h\right) \left( \dashint_{I} h^{1/(1-p)} \right)^{p-1} \leq C
\end{equation}
where the constant $C>0$ depends only on $\alpha$ and $p$.
\end{lemma}

\begin{proof}
Fix an interval $I=(a,b)\subset (q-R,q+R)$; we consider several cases depending on the length and position of $I$. 

Suppose $|b-a|>r/8^{\beta}$. We have the simple estimate
\begin{equation}\label{simpleestimate}
\dashint_{I} h \leq \dashint_{I} f \leq 2L.
\end{equation}
For the second term in \eqref{localizedAp} we use the bounds on $f$ and the fact that $h=f$ outside $(q-r,q+r)$ to see
\[\int_{I}h^{1/(1-p)} \leq \int_{q-r}^{q+r}g^{1/(1-p)} + CL^{1/(1-p)}|I|.\]
Using the fact $p>1+\alpha$ and $r<8^{\beta}|I|$ we can continue,
\begin{align*}
\int_{q-r}^{q+r} g^{1/(1-p)} &= (2L/r^{\alpha})^{1/(1-p)} \int_{0}^{r} |x|^{\alpha/(1-p)}\\
&\leq CL^{1/(1-p)}r\\
&\leq CL^{1/(1-p)}|I|.
\end{align*}
Thus we obtain
\[\left(\dashint_{I}h^{1/(1-p)}\right)^{p-1}\leq CL^{-1}\]
and, by combining this with \eqref{simpleestimate}, we obtain \eqref{localizedAp}.

Now suppose $|b-a|\leq r/8^{\beta}$ and $I\subset [q-(r/4^{\beta}),q+(r/4^{\beta})]$. Then $h=g$ on $I$ and \eqref{localizedAp} follows from Lemma \ref{lem:xalpha}.

Finally suppose $|b-a|\leq r/8^{\beta}$ and $I$ is not strictly contained in the interval $[q-(r/4^{\beta}),q+(r/4^{\beta})]$. This implies that $|x-q|\geq r/4^{\beta}- r/8^{\beta}$ for all $x \in I$; it follows that the values of $g$, and hence the values of $h$, on $I$ are comparable to $L$. In this case the validity of \eqref{localizedAp} is again clear.
\end{proof}

The next lemma will be used to estimate \eqref{Apinequality} for relatively large intervals; the idea is that $w_{k}$ and $w_{k-1}$ agree except on a relatively small interval.

\begin{lemma}\label{increase}
The following estimates hold for both $+$ and $-$:
	\item \[\int_{I_{k}} w_{k} \leq \varepsilon_{k} \int_{J_{k}^{\pm}}w_{k-1},\]
	\item \[\int_{I_{k}} w_{k}^{1/(1-p)} \leq \varepsilon_{k}\int_{J_{k}^{\pm}} w_{k-1}^{1/(1-p)}.\]
\end{lemma}

\begin{proof}
Let $L=w_{k-1}(q_k)$. For the first estimate we note,
	\[\int_{I_{k}} w_{k} \leq 2|I_{k}|w_{k-1}(q_k)=4r_{k}L\]
	and
	\[ \int_{J_{k}^{\pm}}w_{k-1}\geq L/2 (R_{k}-r_k)\]
	so the estimate holds since $R_{k}$ was chosen sufficiently large relative to $r_{k}$.
The argument for the second estimate is similar: we have, since $p>1+\alpha$,
	\[ \int_{I_k} w_k^{1/(1-p)} \leq \int_{-r_k}^{r_k} \left( \Bigl|\frac x {r_k} \Bigr|^{\alpha} L\right)^{1/(1-p)} = 2r_k L^{1/(1-p)} \frac { p-1}{p-1-\alpha}, \] 
	\[ \int_{J^{\pm}_k} w_{k-1}^{1/(1-p)} \geq (2L)^{1/(1-p)} (R_k-r_k) \]
	and again, since $R_k$ are sufficiently large relative to $r_k$, we get the conclusion.
\end{proof}

We now put together Lemma \ref{universal} and Lemma \ref{increase} to obtain the required control on inequality \eqref{Apinequality} for the weights $w_{k}$ used to construct $w$.

\begin{lemma}\label{control}
There exists a constant $C>0$, depending only on $p$ and $\alpha$, such that for all intervals $I$,
\[\left( \dashint_{I}w_{k}\right)\left( \dashint_{I}w_{k}^{1/(1-p)}\right)^{p-1}\leq \max \left\{ (1+\varepsilon_{k})^{p}\left( \dashint_{I}w_{k-1}\right)\left( \dashint_{I}w_{k-1}^{1/(1-p)}\right)^{p-1},C\right\}.\]
\end{lemma}

\begin{proof}
We clearly can assume $I\cap I_{k}\neq \varnothing$ since $w_{k}=w_{k-1}$ outside $I_{k}$. First suppose $|I|>|J_{k}|$ so that (without loss of generality) $J_{k}^{+} \subset I$. Using Lemma \ref{increase} we can estimate,
\begin{align*}
\dashint_{I} w_{k} &= \frac{1}{|I|}\left( \int_{I_{k}}w_{k}+\int_{I\setminus I_{k}}w_{k}\right)\\
&\leq \frac{1}{|I|}\left(\varepsilon_{k}\int_{J_{k}^{+}}w_{k-1}+\int_{I\setminus I_{k}}w_{k-1}\right)\\
&\leq \frac{1}{|I|}\left(\varepsilon_{k}\int_{I}w_{k-1}+\int_{I}w_{k-1}\right)\\
&=(1+\varepsilon_{k})\dashint_{I}w_{k-1}.
\end{align*}
One obtains the estimate 
\[\left(\dashint_{I}w_{k}^{1/(1-p)}\right)^{p-1}\leq (1+\varepsilon_{k})^{p-1}\left(\dashint_{I}w_{k-1}^{1/(1-p)}\right)^{p-1}\]
in exactly the same way. Hence we obtain the desired inequality for this interval $I$.

Next we suppose $|I|\leq |J_{k}|$ so that $I\subset (q_k -4R_k,q_k+4R_k)$. Then, from the construction of $w_{k}$, we have
\[w_{k-1}(q_{k})/2\leq w_{k-1}(x)\leq 2w_{k-1}(q_{k})\]
whenever $|x-q_{k}|\leq 4R_{k}$. By applying Lemma \ref{universal} with $q=q_{k}$, $R=4R_{k}$, $f=w_{k-1}$, $L=w_{k-1}(q_{k})$, $r=r_{k}$ and $g=g_{k}$ we obtain 
\[\left( \dashint_{I}w_{k}\right)\left( \dashint_{I}w_{k}^{1/(1-p)}\right)^{p-1} \leq C\]
with constant $C$ depending only on $p$ and $\alpha$. This proves the claimed inequality.
\end{proof}

By iterating Lemma \ref{control} we can easily show that $w$ is an $A_{p}$-weight on $\mathbb{R}$; combining this with the lattice property of $A_{p}$-weights will then show that $\widehat{w}$ is an $A_{p}$-weight on $\mathbb{R}^{n}$.

\begin{proposition}\label{Apweightprop}
If $p>1+\alpha$ then $\widehat{w}$ is an $A_{p}$-weight on $\mathbb{R}^{n}$.
\end{proposition}

\begin{proof}
By repeated application of Lemma \ref{control} and the fact $\varepsilon_{k}$ can be chosen small we deduce
\[\left( \dashint_{I}w_{k}\right)\left( \dashint_{I}w_{k}^{1/(1-p)}\right)^{p-1}\]
is bounded uniformly in $k$ and $I$. Using the monotone convergence theorem we deduce that
\[\left( \dashint_{I}w\right)\left( \dashint_{I}w^{1/(1-p)}\right)^{p-1}\]
is bounded uniformly in $I$. This shows that $w$ is an $A_{p}$-weight on $\mathbb{R}$.

We now observe that
\[x=(x_{1},\dots,x_{n})\mapsto \eta_{i}(x):=w(x_{i})\]
is an $A_{p}$-weight on $\mathbb{R}^{n}$ for each $1\leq i\leq n$. Indeed; we may use cubes instead of Euclidean balls in the left hand side of \eqref{Apinequality} and then the left hand side of \eqref{Apinequality}, corresponding to the weight $\eta_{i}$, reduces to the corresponding expression for the weight $w$ on $\mathbb{R}$. Such an expression is obviously bounded since $w$ is an $A_{p}$-weight on $\mathbb{R}$. 

By Proposition 4.3 \cite{KKM} the minimum of a finite collection of $A_{p}$-weights is again an $A_{p}$-weight; hence $\widehat{w}=\min \{\eta_{1},\dots,\eta_{n} \}$ is an $A_{p}$-weight.
\end{proof}

Taken together, Proposition \ref{trivialmodulusprop} and Proposition \ref{Apweightprop} prove Theorem \ref{mainthm}.

\section{Characterization of the weak gradient on $\mathbb{R}$}

Let $\mu$ be a locally finite Borel measure on $\mathbb{R}$. We give a characterization of the $p$-weak gradient for Lipschitz functions defined on $(\mathbb{R},|\cdot|,\mu)$. The idea is that integrability properties of the absolutely continuous part of $\mu$ give information about which intervals (considered as curves) have non trivial $p$-modulus; these intervals then determine the $p$-weak gradient. A similar characterization has been found in \cite{bjbj}, for measures $\mu$ whose absolutely continuous part with respect to Lebesgue measure is bounded by below by a constant, and a weaker result is stated in \cite{Bog}, Theorem 2.6.4, where the author characterize the measures for which the $p$-weak gradient is $|f'|$ for every $f \in C^{\infty}$ (which is equivalent to the closability of the Sobolev norm he considers).

It is worth noticing that, at least when $p=2$, a very similar question has been investigated by some authors in the calculus of variations, posed as a semicontinuity problem; in \cite{Marcell, FusMos} they found exactly the same answer that we find.

Throughout this section we fix $p>1$ and let $q$ be the corresponding H\"older conjugate so that $p^{-1}+q^{-1}=1$. Given a compact interval $I\subset \mathbb{R}$ we define the corresponding curve $\gamma_{I}\colon I \to \mathbb{R}$ by $\gamma_{I}(t)=t$. Denote the Lebesgue decomposition of $\mu$ by $\mu = \mu_a + \mu_s$. Let $\mu_a=f_a \mathcal{L}^1$ with $f_{a}\colon \mathbb{R} \to \mathbb{R}$ a Borel function and fix a Lebesgue null set $N\subset \mathbb{R}$ on which $\mu_s$ is concentrated.

\begin{lemma}\label{lem:modloc} For any interval $[a,b]\subset \mathbb{R}$ we have ${\rm Mod}_{p,\smu} \bigl( \{\gamma_{[a,b]}\} \bigr) >0 $ if and only if $f_a^{1/(1-p)}$ is Lebesgue integrable on $[a,b]$.
\end{lemma}

\begin{proof} This lemma is an easy corollary of Theorem $5.1$ in \cite{ADMS}; however we want to give here a self-contained and more elementary proof since $\Gamma$ consists of only one curve. If $a=b$ the statement is trivial so we assume $a<b$.

We write an equivalent definition for ${\rm Mod}_{p,\smu}$, using the homogeneity of the problem (see \cite{ADMS}):
\begin{equation}\label{eqn:equiv}
{\rm Mod}_{p,\smu} ( \{\gamma_{[a,b]}\} )^{1/p} = \inf \left\{ \frac { \| g \|_{L^p(\mu)} }{  \int_a^b g(x) \dd x} \right\},
\end{equation}
where the infimum is taken over all Borel functions $g$ which are $p$-integrable with respect to $\mu$ (this set is non empty since $\mu$ is locally finite).

Let $g\colon \mathbb{R} \to \mathbb{R}$ be any Borel function. From H\"{o}lder's inequality we have
\begin{equation}\label{eqn:ineq}
\int_a^b g(x) \dd x \leq \left( \int_a^b g^p(x) f_{a}(x)  \dd x \right)^{1/p}\left(  \int_a^b f_{a}(x)^{1/(1-p)}  \dd x \right)^{1/q}.
\end{equation}
Now, if $f_a^{1/(1-p)}$ is $\mathcal{L}^1$ integrable on $[a,b]$, by using inequality \eqref{eqn:ineq} in \eqref{eqn:equiv} we get that
$${\rm Mod}_{p,\smu} ( \{\gamma_{[a,b]}\} )^{1/p} \geq \inf \left\{ \frac { \| g \|_{L^p(\mu_a)} }{  \int_a^b g(x) \dd x} \right\}  \geq \frac 1{ \| f_a^{1/(1-p)} \|^{1/q}_{L^1(\mathcal{L}^1)} } > 0 .$$
If otherwise $f_a^{1/(1-p)}$ is not integrable then, letting $f_{\varepsilon} =\max \{f_a,\varepsilon\}$, we use
$$g(x)=\begin{cases} 0 & \text{ if }x \in N \cup (\mathbb{R}\setminus [a,b]) \\ f_{\varepsilon}^{1/(1-p)}(x) & \text{ otherwise} \end{cases}$$
as a test function in $\eqref{eqn:equiv}$ and using $\mu_a \leq f_{\varepsilon} \mathcal{L}^1$ we get
$$ {\rm Mod}_{p,\smu} (\{ \gamma_{[a,b]} \})^{1/p} \leq  \left( \int_a^b f_{\varepsilon}^{1/(1-p)}(x) \dd x \right) ^{-1/q}.$$
Letting $\varepsilon \to 0$ we obtain, by monotone convergence, that ${\rm Mod}_{p,\smu}(\{ \gamma_{[a,b]} \})=0$.
\end{proof}

\begin{theorem}\label{thm:char}Let
 \begin{equation}
  \mathcal{N}_p =  \left\{ x \in \mathbb{R} \text{ such that } f_a^{1/(1-p)} \text{ is integrable on a neighbourhood of }x \right\}.
 \end{equation}
Let $f\colon \mathbb{R} \to \mathbb{R}$ be Lipschitz and define, for $\mu$ almost every $x$,
 \begin{equation}\label{eqn:dfnwgrad}
  	| df |_{p,\smu} (x) = \begin{cases} | f'(x) | & \text{ if }x \in \mathcal{N}_p \setminus N \\
  	                     0 & \text{otherwise.} \end{cases}
 \end{equation}
 Then $| \nabla f|_{p,\smu}(x) = | df |_{p,\smu}(x)$ for $\mu$-almost every $x$.
\end{theorem}

\begin{proof} We first note that equation \eqref{eqn:dfnwgrad} makes sense because $f'$ exists $\mathcal{L}^1$-almost everywhere, by Rademacher theorem, and so it exists also $\mu_a$-almost everywhere; hence $f'$ exists $\mu$-almost everywhere in the complement of $N$. We note that, thanks to Lemma \ref{lem:modloc}, we have the following equivalent definition for $\mathcal{N}_p$:
\begin{equation}\label{eqn:defnp}
  \mathcal{N}_p = \bigcup_{ \varepsilon > 0 } \left\{ x \in \mathbb{R} \text{ such that } {\rm Mod}_{p,\smu}\bigl( \{ \gamma_{[x-\varepsilon,x+\varepsilon]} \}\bigr) > 0 \right\}.
\end{equation}
Denote by $B$ the set of points where $f$ is not differentiable. Set 
$$G_f = \{ g: \mathbb{R} \to [0,\infty) \text{ bounded Borel function} \; : \; g(x) \geq | f'(x)| \text{ for $\mathcal{L}^1$-a.e. }x \in \mathcal{N}_p \}.$$
We will prove that $G_f$ is exactly the set of bounded $p$-upper gradients for $f$. This implies the theorem: indeed, $|df|_{p,\smu} \in G_f$ and for any $g \in G_f$ we have that $g(x) \geq |df|_{p,\smu}(x)$ for $\mu$ almost every $x \in \mathbb{R}$.\\

\textbf{ Step 1.} $g$ a bounded $p$-upper gradient $\Longrightarrow$ $g \in G_f$. \\ Let $D_p$ be the set of Lebesgue points of $g$ with respect to the Lebesgue measure. Since $g$ is a bounded Borel function, we know that $\mathcal{L}^1(D_p^c)=0$. Now take a point $x \in \mathcal{N}_p \cap D_p \setminus (B \cup N) $. Thus there exists $\varepsilon$ such that ${\rm Mod}_{p,\smu}\bigl(\{ \gamma_{[x-\varepsilon,x+ \varepsilon]}\} \bigr) > 0$; but then ${\rm Mod}_{p,\smu}\bigl( \{\gamma_{[x-\delta,x+\delta]}\} \bigr) > 0$ for every $0<\delta\leq\varepsilon$. This, together with the definition of the $p$-upper gradient, gives us that
$$ |f(x+\delta)-f(x-\delta)| \leq \int_{x-\delta}^{x+\delta} | \nabla f |_{p, \smu}(s) \dd s,$$
and so, passing to the limit when $\delta \to 0$, we get that $|f'(x)| \leq g(x)$, and so the thesis.\\

\textbf{ Step 2.} $g \in G_f$ $\Longrightarrow$ $g$ is a $p$-upper gradient. \\ To prove this implication we first show that
$$\Gamma=\{ \gamma \; : \;\gamma \text{ has end points }a<b,\, (a,b) \cap \mathcal{N}_p^c \neq \emptyset \}$$
is ${\rm Mod}_{p,\smu}$-null. Let $\mathcal{B}_p=\mathcal{N}_p^c$.
First let $\{x_n\}_{n \in \mathbb{N}} \subset \mathcal{B}_p$ be a set of points dense in $\mathcal{B}_p$. From the definition of $\mathcal{N}_p$ we know that for every $n$ there exists a non negative function $f_n \in L^p(\mathbb{R}, \mu)$ such that $f_n$ is not locally Lebesgue integrable at $x_n$, that is:
\begin{equation}\label{eqn:epsi}
\int_{x_n- \varepsilon}^{x_n+\varepsilon} f_n(s) \dd s = \infty \qquad \forall \varepsilon >0.
\end{equation}
Now we take $f= \sum_n a_n f_n$ where the $a_n$ are positive real numbers small enough so that $f$ belongs to $L^p(\mathbb{R},\mu)$. For every curve $\gamma \in \Gamma$ with end points $a<b$ we have that $x_n \in (a,b)$ for some $n$ (since $\{x_{n}\}_{n \in \mathbb{N}}$ were dense in $\mathcal{B}_p$) and so we have that $[x_n- \varepsilon,x_n+ \varepsilon] \subset (a,b)$ for $\varepsilon>0$ small enough. In particular, using \eqref{eqn:epsi},
$$\int_{\gamma} f \geq \int_a^b f(s) \dd s \geq a_n \int_a^b f_n(s) \dd s \geq a_n \int_{x_n- \varepsilon}^{x_n+ \varepsilon} f_n(s) \dd s = \infty $$
and so ${\rm Mod}_{p,\smu}(\Gamma)=0$.\\
Suppose $g\in G_{f}$ and $\gamma \notin \Gamma$ has end points $a<b$. Then $(a,b)\subset \mathcal{N}_{p}$ and hence,
$$|f(a) - f(b)| \leq \int_{a}^{b} |f'(x)| \dd x \leq \int_{a}^{b} g(x) \dd x \leq \int_{\gamma} g.$$
Thus the set of curves where the upper gradient property fails is a $p$ negligible set; therefore $g$ is a $p$-upper gradient of $f$.\\

\end{proof}

\begin{remark}
It seems that one can generalize the observations in section 5 about weak gradients on $\mathbb{R}$ to analogous statements about $\mathbb{R}^{n}$; the statement here should be that the weak gradient at a point is the restriction to a subspace (depending on the point and the measure) of the ordinary derivative. This generalization involves the equivalent definition of weak gradient from \cite{Che99} as an integrand whose integral represents the Cheeger energy. The Cheeger energy is a functional obtained by relaxing the integral of the slope using convergence of Lipschitz functions; the paper \cite{Bou} provides integral representations of many such functionals. Unfortunately, when $n>1$, apart from peculiar cases, it is not possible to give a concrete description of the subspaces but a rather abstract one.
\end{remark}

\end{document}